\documentclass{article}
\usepackage{graphicx}
\usepackage[english]{babel}
\usepackage{amssymb,amsmath, amsthm, amscd,ifthen}
\usepackage{xcolor}
\usepackage{dsfont}

\newtheorem{conj}{Conjecture}
\newtheorem{theorem}{Theorem}
\newtheorem{definition}{Definition}
\newtheorem{lemma}[theorem]{Lemma}
\newcommand{\m}{\mathcal}

\title{More on the Erd\H os--Kleitman problem on matchings in set families}
\author{Andrey Kupavskii\footnote{Moscow Institute of Physics and Technology, St. Petersburg State University, Innopolis University, Russia; Email: {\tt kupavskii@ya.ru}}, Georgy Sokolov\footnote{Moscow Institute of Physics and Technology, Innopolis University, Russia; Email: {\tt sokolov.gm@phystech.edu}}}
\date{}

\begin{document}

\maketitle

\begin{abstract}
Let $e(n,s)$ denote the maximum size of a family $\mathcal{F}$ of subsets of an $n$-element set that contains no $s$ pairwise disjoint members. In 1968, answering a question of Erd\H{o}s, Kleitman determined $e(sm-1,s)$ and $e(sm,s)$ for all integers $m,s\ge 1$. Half a century later, Frankl and Kupavskii determined $e(s(m+1)-\ell, s)$ for $\ell \leq \frac{s-3}{m+3}$. They showed that the corresponding extremal example is closely connected with the extremal example for the Erd\H{o}s Matching Conjecture, and conjectured that  the same remains true for all $\ell \leq s/2$. In this paper, we prove an approximate version of their conjecture for $s\ge s_0(m)$.
\end{abstract}

\section{Introduction}

Let $[n] = \{1, \ldots, n\}$. Any collection of sets is called a family. Given a set $X$, we denote by $2^{X}$, ${X \choose k}$, and ${X \choose \leq k}$ the families of all subsets of $X$, subsets of size exactly $k$, and subsets of size at most $k$, respectively. We refer to the family ${[n] \choose k}$ as the {\it $k$-th layer of the Boolean lattice}. A (full) {\it star} is the family of all sets in ${X \choose k}$ containing a fixed element.

A {\it matching} is a family of pairwise disjoint sets. Given a family $\mathcal{F}$, its {\it matching number} $\nu(\mathcal{F})$ is the maximum size of a matching in $\mathcal{F}$.
Let us introduce two classical extremal quantities concerning families with bounded matching number. Both were introduced by Erd\H{o}s in the 1960s:
\begin{align*}
    e(n,s)&=\max\big\{|\mathcal{F}|: \mathcal{F}\subset 2^{[n]}, \nu(\mathcal{F})<s\big\},\\
    e_k(n,s)&=\max\Big\{|\mathcal{F}|: \mathcal{F}\subset {[n]\choose k}, \nu(\mathcal{F})<s\Big\}.
\end{align*}

This paper is devoted to the study of the quantity $e(n, s)$. However, these two quantities are closely connected, so we briefly discuss $e_k(n, s)$ first.

Obviously, we have $e_k(n, s) = {n \choose k}$ for $n < sk$. The famous Erd\H{o}s--Ko--Rado theorem \cite{EKR} states that $e_k(n, 2) = {n - 1 \choose k - 1}$ for $n \geq 2k$. Soon after the EKR paper, Erd\H{o}s made a general conjecture concerning the value of $e_k(n, s)$ for all $n, k, s$. Let us define several families of $k$-element subsets of $[n]$:
\begin{equation}\label{eq: frankl families}
\mathcal{A}_i^{(k)}(n,s) := \Bigl\{F\in {[n]\choose k}:|F\cap [si-1]|\ge i\Bigr\}, \qquad 1\le i\le k.
\end{equation}

It is easy to check that these families do not contain an $s$-matching. Therefore,
\begin{equation} \label{eq: EMC examples}
e_k(n, s) \geq \max\left\{|\mathcal{A}_i^{(k)}(n,s)| : 1 \leq i \leq k\right\}.
\end{equation}
A careful analysis of the values $|\mathcal{A}_i^{(k)}(n,s)|$ shows that the maximum in \eqref{eq: EMC examples} is always attained for $i = 1$ or $i = k$. The Erd\H{o}s Matching Conjecture states that these examples are best possible.

\begin{conj}[Erd\H{o}s Matching Conjecture \cite{E}]\label{conj: EMC}
For $n\ge sk$,
\begin{equation}\label{eq: EMC statement}
e_k(n,s) = \max\left\{|\mathcal{A}_1^{(k)}(n,s)|,|\mathcal{A}_k^{(k)}(n,s)|\right\}.
\end{equation}
\end{conj}

Despite significant effort, the Erd\H{o}s Matching Conjecture remains open in general. It has been proved for $k\le 3$ \cite{EG, F11}. The maximum in \eqref{eq: EMC statement} is attained at the first term for $n$ sufficiently large compared with $sk$, and at the second term for $n$ close to $sk$. It is known that for $n$ sufficiently large we indeed have $e_k(n, s) = |\mathcal{A}_1^{(k)}(n,s)|$, while for $n$ sufficiently close to $sk$ we have $e_k(n, s) = |\mathcal{A}_k^{(k)}(n,s)|$. The best known bounds are due to \cite{F4} and \cite{FK21}. In the latter paper, the authors proved the conjecture for $n>\frac 53sk$ and $s\ge s_0$.

\begin{theorem}[Frankl and Kupavskii~\cite{FK21}] \label{t: Frankl-Kupavskii EMC}
    There exists an absolute constant $s_0$ such that
    \begin{equation}\label{eq: EMC big n}
        e_k(n,s) = |\mathcal{A}_1^{(k)}(n,s)| = {n \choose k} - {n - s + 1 \choose k},
    \end{equation}
    provided
    \[
    n \geq \frac{5}{3}(s-1)k - \frac{2}{3}(s-1)
    \qquad\text{and}\qquad
    s \geq s_0.
    \]
\end{theorem}

The conjecture is also known to hold for $n$ very close to $sk$ and $s\ge s_0(k)$, see \cite{F7, KoKu}.

For $n \geq 2ks(1+o(1))$, Frankl and Kupavskii proved a Hilton--Milner-type stability result \cite{FK6}. We will need it in the proof of our main result, and we state it precisely in Section~\ref{sec3}. Informally, it says that in this regime of parameters, the largest nontrivial family with $\nu(\mathcal{F})\le s-1$ is a union of $s-2$ stars and a Hilton--Milner family.

We now turn to the quantity $e(n, s)$. It turns out that for $s \geq 3$ the behaviour of $e(n, s)$ depends on the residue of $n$ modulo $s$. In what follows, we use the parametrization
\[
n = sm + s - \ell,
\qquad 1 \le \ell \le s.
\]
In 1968, answering a question of Erd\H{o}s, Kleitman determined $e(sm+s-\ell,s)$ for $\ell \in \{1, s\}$.

\begin{theorem}[Kleitman~\cite{Kl}]  \label{t: Kleitman}
\begin{align}
\label{eq: EK l=1} e(sm-1,s) &= \sum_{t=m+1}^{sm+s-1}{sm+s-1\choose t},\\
\label{eq: EK l=0} e(sm,s) &= {sm-1\choose m}+\sum_{t=m+1}^{sm}{sm\choose t}.
\end{align}
\end{theorem}

The family achieving equality in \eqref{eq: EK l=1} is simply ${[n] \choose \geq m + 1}$. It is not difficult to check that $e(sm, s) = 2e(sm-1, s)$. A family achieving equality in \eqref{eq: EK l=0} may be obtained from ${[n] \choose \geq m}$ by the {\it doubling operation}. Given a family $\mathcal{F} \subset 2^{[n]}$, its {\it doubling} is the family $\mathcal{F}' \subset 2^{[n+1]}$ defined by
\[
\mathcal{F}' = \{F \subset [n + 1]: F \cap [n] \in \mathcal{F}\}.
\]
It is easy to check that $|\mathcal{F}'| = 2|\mathcal{F}|$ and that $\nu(\mathcal{F}) < s$ implies $\nu(\mathcal{F}') < s$.

Note that the extremal examples in both cases of Theorem~\ref{t: Kleitman} may be described in the following way: the extremal family consists of all sets of size at least $m+1$, together with a largest family of $m$-element sets that does not contain an $\ell$-matching. In the case $\ell=1$ the latter family is empty, while in the case $\ell=s$ it is $\mathcal{A}_m^{(m)}(n,s)$. In \cite{FK9}, Frankl and Kupavskii introduced the following family.

\begin{definition}\label{def3}
Let $n = sm+s-\ell$, with $0<\ell\le s$. Set
\[
\mathcal{P}(m,s,\ell) := \bigl\{P\subset [n]: |P|+|P\cap [\ell-1]|\ge m+1\bigr\}.
\]
\end{definition}

This family consists of all sets of size at least $m+1$, the family $\mathcal{A}_1^{(m)}(n, \ell)$, and some sets of size less than $m$. It is easy to check that $\nu(\mathcal{P}(m,s,\ell)) < s$. In \cite{FK9}, Frankl and Kupavskii made the following conjecture.

\begin{conj} \label{conj: EK}
    Suppose that $s\ge 2$, $m\ge 1$, and $n = sm+s-\ell$ for some integer $0<\ell\le \lceil \frac s2\rceil$. Then
    \begin{equation}\label{eq007}
        e(sm+s-\ell,s) = |\mathcal{P}(m,s,\ell)|.
    \end{equation}
\end{conj}

The condition $\ell\le \lceil \frac s2\rceil$ is needed to ensure that $\mathcal{A}_1^{(m)}(n, \ell)$ is the largest family of $m$-element sets avoiding an $\ell$-matching.
In \cite{FK9} and \cite{FK8}, Conjecture~\ref{conj: EK} was proved in several cases.

\begin{theorem} \label{t: Frankl-Kupavskii EK}
The equality $e(sm+s-\ell,s) = |\mathcal{P}(m,s,\ell)|$ holds in each of the following cases:
\begin{align*}
    &\mathrm{(i)}\qquad \ell = 2, \\
    &\mathrm{(ii)}\qquad m=1,\\
    &\mathrm{(iii)}\qquad s\ge \ell m+3\ell+3.
\end{align*}
\end{theorem}

Recently, Kupavskii and Sokolov determined $e(2s + s - \ell, s)$ for all $s$ and $\ell$ \cite{KS}. In particular, Conjecture~\ref{conj: EK} was confirmed for $m=2$. We refer to that paper for a more thorough and up-to-date discussion of the problem.

In this paper, we prove an approximate version of this conjecture for $s$ sufficiently large compared to $m$.

\begin{theorem} \label{t: main}
    Let $n = sm + s - \ell$. For every integer $m$ and every $\varepsilon > 0$ there exists $s_0$ such that
    \begin{equation}\label{eq: main theorem}
        e(n,s) = |\mathcal{P}(m,s,\ell)|,
    \end{equation}
    provided $s > s_0$ and $\ell \leq \left(\frac{1}{2} - \varepsilon\right)s$.
\end{theorem}

We should note that very recently, the authors of the present paper have made a general conjecture concerning the values of $e(n,s)$ for all $n,s$ and verified it for large $s$ `from the other end', i.e., for $n=ms+c$ and $s>s_0(c)$, see \cite{KS2}.

Given a family $\mathcal{F}$, we denote by $\mathcal{F}^{(k)}$ ($\mathcal{F}^{(\leq k)}$, $\mathcal{F}^{(\geq k)}$) the family of all sets in $\mathcal{F}$ of size exactly $k$ (at most $k$, at least $k$). In fact, we prove a stronger variant of Theorem~\ref{t: main}: every family avoiding an $s$-matching has at most as many sets in the first $m+1$ layers of the Boolean lattice as $\mathcal{P}(m,s,\ell)$. More precisely, we prove the following theorem.

\begin{theorem} \label{t: strong main}
    For every integer $m$ and every $\varepsilon > 0$ there exists $s_0$ such that the following holds. Put $n=sm+s-\ell$. If $s > s_0$, $\ell \leq \left(\frac{1}{2} - \varepsilon\right)s$, and $\mathcal{F} \subset 2^{[n]}$ is a family with $\nu(\mathcal{F}) < s$, then
    \[
    |\mathcal{F}^{(\leq m+1)}| \leq |\mathcal{P}(m,s,\ell)^{(\leq m+1)}|.
    \]
\end{theorem}

Theorem~\ref{t: strong main} obviously implies Theorem~\ref{t: main}, since $\mathcal{P}(m,s,\ell)$ contains all sets from ${[n] \choose \geq m + 2}$. We note that in \cite{FK9} Frankl and Kupavskii conjectured that the statement of Theorem~\ref{t: strong main} should hold, but managed to prove it, under the assumptions of Theorem~\ref{t: Frankl-Kupavskii EK}, only for the first $m+2$ layers of the Boolean lattice. Theorem~\ref{t: strong main} confirms their conjecture. We also note that if one replaces $m+1$ by $m$ in its statement, then the resulting claim is false.

\section{Overview of the proof}
We follow essentially the same overall strategy as in point~(iii) of Theorem~\ref{t: Frankl-Kupavskii EK}, but in several places we develop and use a more refined analysis. We work in the regime where $s$ is large. In this regime, the quantity $\sum_{i=0}^{m-1}{n \choose i}$ is small compared to ${n\choose m}$. At the same time, in our parametrization the $m$-th layer is the largest layer of the Boolean lattice that may contain an $s$-matching. For these reasons, the first and main part of the proof focuses on layers $m$ and $m+1$.

The family $\mathcal{P}(m, s, \ell)$ contains all $(m+1)$-element sets and, on the $m$-th layer, the largest family that does not contain an $\ell$-matching (in the regime under consideration, this family is a union of $\ell-1$ stars). If an extremal family $\mathcal{F}$ with $\nu(\mathcal{F})<s$ contains more sets than $\mathcal{P}(m, s, \ell)$ on these two layers, then it must contain an $\ell$-matching on the $m$-th layer. Consequently, $\mathcal{F}$ cannot contain all $(m+1)$-element sets. In fact, a simple averaging argument shows that $\mathcal{F}$ must miss many $(m+1)$-element sets. This is the crucial part of the proof: we need to find the best trade-off between the additional $m$-sets that we gain and the $(m+1)$-sets that we lose. This is also the key point at which our approach differs from that of \cite{FK9} and allows us to replace the condition $s \geq \ell m + 3\ell + 3$ by the much weaker condition $s \geq 3\ell$, for sufficiently large $s$.

The conclusion of the first step is that an extremal family $\mathcal{F}$ with $\nu(\mathcal{F}) < s$ and $|\mathcal{F}| \geq |\mathcal{P}(m, s, \ell)|$ cannot contain an $\ell$-matching on the $m$-th layer. Using the stability result of Frankl and Kupavskii (Theorem~\ref{t: Hilton-Milner for EMC}), we then deduce that if $\mathcal{F}^{(m)}$ is not contained in a union of $\ell-1$ stars, then $|\mathcal{F}^{(m)}|$ is substantially smaller than $|\mathcal{A}_{1}^{(m)}(n, \ell)|$. This loss is large enough to compensate for any possible gains on the layers of size at most $m-1$.

Finally, we analyze the lower layers of the Boolean lattice. We show that if $\mathcal{F}$ is not a subfamily of $\mathcal{P}(m, s, \ell)$, then it must miss either many sets from $\mathcal{P}(m, s, \ell)^{(m)}$ or many sets from ${[n] \choose m + 1}$, and therefore it is smaller than $\mathcal{P}(m, s, \ell)$.

\section{Preliminaries and auxiliary calculations}\label{sec3}

Recall the notion of shifting. Let $A = \{a_1, \ldots, a_k\}$ and $B = \{b_1, \ldots, b_k\}$ be $k$-element subsets of $[n]$, with their elements listed in increasing order. We say that $A$ can be shifted to $B$ if $a_i \geq b_i$ for all $i \in [k]$. A family $\mathcal{F}$ is {\it shifted} if $A \in \mathcal{F}$ implies $B \in \mathcal{F}$ for every $B$ such that $A$ can be shifted to $B$. It is well known \cite{F3} that it is enough to prove Theorem~\ref{t: strong main} for shifted families.

Next, we state the Hilton--Milner-type stability result of Frankl and Kupavskii \cite{FK6}. Let us first define the relevant family:
\begin{align*}
    \mathcal{H}^{(k)}(n, s)
    &= \Big\{H \in {[n] \choose k}: H \cap [s-2] \neq \emptyset \Big\} \cup \{[s, s+k-1]\} \\
    &\cup  \Big\{H \in {[s-1,n] \choose k}: s-1\in H,\ H \cap [s, s+k-1] \ne \emptyset \Big\}.
\end{align*}

The family $\mathcal{H}^{(k)}(n, s)$ is the union of $s-2$ stars with centers in $[s-2]$ and the Hilton--Milner family on $[s-1,n]$. Clearly, we have $\nu(\mathcal{H}^{(k)}(n, s)) < s$.

\begin{theorem}[Frankl and Kupavskii~\cite{FK6}] \label{t: Hilton-Milner for EMC}
    Let $n \geq 2sk(1+o(1))$, where $o(1)$ is taken with respect to $s \to \infty$. Assume that $\mathcal{G}\subset {[n]\choose k}$ satisfies $\nu(\mathcal{G}) < s$ and, additionally, that $\mathcal{G}$ is not isomorphic to a subfamily of $\mathcal{A}_1^{(k)}(n,s)$. Then
    \[
    |\mathcal{G}| \leq |\mathcal{H}^{(k)}(n, s)| = |\mathcal{A}_1^{(k)}(n,s)| - {n - k - s + 1 \choose k - 1} + 1.
    \]
\end{theorem}

We emphasize that any improvement of the constant $2$ in Theorem~\ref{t: Hilton-Milner for EMC} would imply an improvement of the constant $\frac{1}{2}$ in the condition $\ell \leq \frac{s}{2}$ in Theorems~\ref{t: main} and~\ref{t: strong main}.

We will also need the following universal bound on $e_k(n, s)$.

\begin{lemma}[Frankl~\cite{F3}] \label{l.frankl_shifting}
    Let $n \geq ks$. Then
    \[
    e_k(n, s) \leq (s-1){n - 1 \choose k - 1}.
    \]
\end{lemma}

It will be more convenient for us to use Lemma~\ref{l.frankl_shifting} in the following slightly modified form. For a family $\m F\subset 2^{[n]}$, we denote by $y_{\mathcal{F}}(k)$ the number of $k$-element sets that are not in $\mathcal{F}$. That is, $y_{\mathcal{F}}(k) = {n \choose k} - |\mathcal{F}^{(k)}|$.

\begin{lemma} \label{l.frankl_shifting modified}
    Let $\mathcal{F} \subset 2^{[n]}$ be a family without an $s$-matching, and let $k, t \geq 0$ be integers such that $n = ks + t$. Then
    \[
        y_{\mathcal{F}}(k) \geq \left(1+\frac{t}{k}\right){n-1 \choose k - 1}.
    \]
\end{lemma}

\begin{proof}
    By Lemma~\ref{l.frankl_shifting} we have
    \[
    |\mathcal{F}^{(k)}| \leq e_k(n, s) \leq (s-1){n-1 \choose k - 1},
    \]
    and therefore
    \begin{align*}
        y_{\mathcal{F}}(k)
        &= {n \choose k} - |\mathcal{F}^{(k)}|\geq {n \choose k} - (s-1){n-1 \choose k - 1} \\
        &= \left(\frac{n}{k}-s+1\right){n-1 \choose k - 1} = \left(1+\frac{t}{k}\right){n-1 \choose k - 1}.
    \end{align*}
\end{proof}

In the remainder of this section, we provide several useful calculations with binomial coefficients that will be used later. The following lemma quantifies the fact that the lower layers of the Boolean lattice are small.

\begin{lemma} \label{l: low layers are small}
    For $n \geq sk - 1$ we have
    \[
    \sum_{i=0}^{k}{n \choose i} \leq \frac{s-1}{s-2}{n \choose k}.
    \]
\end{lemma}

\begin{proof}
    For $i \leq k$ we have
    \[
    {n \choose i - 1} = \frac{i}{n-i+1}{n \choose i} \leq \frac{k}{n-k+1}{n \choose i} \leq \frac{1}{s-1}{n \choose i}.
    \]
    Therefore,
    \[
    \sum_{i=0}^{k}{n \choose i}
    \leq \sum_{i=0}^{k}\left(\frac{1}{s-1}\right)^{k-i}{n \choose k}
    \leq \sum_{j=0}^{\infty}\left(\frac{1}{s-1}\right)^{j}{n \choose k}
    =
    \frac{s-1}{s-2}{n \choose k}.
    \]
\end{proof}

The next lemma will be used in the stability part of the proof: we will compare losses in the $m$-th level of the Boolean lattice with gains in the lower layers relative to $\mathcal{P}(s, m, \ell)$.

\begin{lemma} \label{l: HM calculation}
    Let $n = sm + s - \ell$. If $\ell \geq 2$ and $s \geq \frac{3}{2}\ell + m$, then
    \begin{equation} \label{eq: HM calculation}
        {n - m - \ell + 1 \choose m -1} > \frac{\ell}{2}\sum_{i=1}^{m-1}{n-1 \choose i - 1}.
    \end{equation}
\end{lemma}

\begin{proof}
    By Lemma~\ref{l: low layers are small}, the right-hand side of \eqref{eq: HM calculation} is at most
    \[
    \frac{\ell}{2}\cdot\frac{s-1}{s-2}{n - 1 \choose m - 2}
    = \frac{\ell(s-1)(m-1)}{2(s-2)n}{n \choose m - 1}.
    \]
    The left-hand side of \eqref{eq: HM calculation} is
    \begin{align*}
        {n - m - \ell + 1 \choose m - 1}
        &= {n \choose m - 1}\prod_{i=0}^{m-2}\frac{n-m-\ell+1-i}{n-i} \\
        &= {n \choose m - 1}\prod_{i=0}^{m-2}\left(1-\frac{m+\ell-1}{n-i}\right) \\
        &\geq {n \choose m - 1} \left(1-\frac{m+\ell-1}{n-m+2} \right)^{m-1} \\
        &\geq {n \choose m - 1}\left(1 - \frac{(m-1)(m+\ell-1)}{n-m+2} \right).
    \end{align*}
    Thus, it is enough to prove
    \begin{equation} \label{eq: HM calculation 2}
        \frac{\ell(s-1)(m-1)}{2(s-2)n} < 1 - \frac{(m-1)(m+\ell-1)}{n-m+2}.
    \end{equation}
    By the assumptions of the lemma we have
    \[
    n - m + 2 = sm + s - \ell - m + 2 > sm
    \]
    and
    \[
    \frac{\ell}{2(s-2)} \leq \frac{\ell}{3\ell-4} \leq 1.
    \]
    Therefore,
{\small
    \begin{align*}
        &\frac{\ell(s-1)(m-1)}{2(s-2)n} + \frac{(m-1)(m+\ell-1)}{n-m+2} \\
        &\leq \frac{\ell(s-1)(m-1)}{2(s-2)sm} + \frac{(m-1)(m+\ell-1)}{sm} \\
        &< \frac{\ell(s-1)}{2(s-2)s} + \frac{m+\ell-1}{s} \\
        &= \frac{1}{s}\left(\frac{\ell}{2} + \frac{\ell}{2(s-2)} + m + \ell - 1\right) \\
        &\leq \frac{1}{s}\left(\frac{3}{2}\ell + m\right) \leq 1.
    \end{align*}
}
    We conclude that \eqref{eq: HM calculation 2} holds.
\end{proof}

\section{Matchings on the $m$-th layer}

Let $\mathcal{F} \subset 2^{[n]}$ be a shifted family without an $s$-matching, and assume that
\[
|\mathcal{F}^{(\leq m + 1)}| \geq |\mathcal{P}(m, s, \ell)^{(\leq m + 1)}|.
\]
To prove Theorem~\ref{t: strong main}, we need to show that $\mathcal{F} = \mathcal{P}(m, s, \ell)$. In this section we prove that $\mathcal{F}^{(m)}$ does not contain an $\ell$-matching.



\begin{lemma} \label{l: matchings on layer m}
    For every integer $m$ there exists $s_0$ such that for all $s \geq s_0$ and all $\ell \leq \frac{3}{5}s$ the following holds. Let $\mathcal{F} \subset 2^{[n]}$ be a family avoiding an $s$-matching and satisfying
    \[
    |\mathcal{F}^{(\leq m + 1)}| \geq |\mathcal{P}(m, s, \ell)^{(\leq m + 1)}|.
    \]
    Then $\nu(\mathcal{F}^{(m)}) < \ell$. Moreover, for any $k \leq \ell$ there is no $k$-matching $\{F_1, \ldots, F_k\}$ in  $\mathcal{F}$ such that
    \[
    \sum_{i=1}^k |F_i| \leq km + k - \ell.
    \]
\end{lemma}

We will use the moreover part in the analysis of smaller sets in Section~\ref{s: completing proof}.
Lemma~\ref{l: matchings on layer m} is a consequence of the following technical lemma.

\begin{lemma} \label{l: parametrised matchings}
    Let $n = sm+s-\ell$. Assume that there exists $t \geq 0$ such that the following three conditions hold:
    {\small
    \begin{equation}
       \label{eq: condition t 1}
        {(m+1)(s-\ell)-1 \choose m} > (\ell+t-1)\sum_{i=1}^{m-1}{n-1 \choose i-1} + {n - \ell + 1 \choose m} - {n - \ell - t \choose m}.
    \end{equation}
    }
    \begin{equation}\label{eq: condition t 2}
        \left(\frac{t}{m+1} + 1\right){n - (\ell+t)m-1 \choose m} > (s-1)\sum_{i=1}^{m}{n-1 \choose i-1}.
    \end{equation}
    \begin{equation}\label{eq: condition t 3}
        e_{m}(n, \ell + t) = {n \choose m} - {n - \ell - t + 1 \choose m}.
    \end{equation}
    Then for any $\mathcal{F}$ with $\nu(\mathcal{F}) < s$ and
    \[
    |\mathcal{F}^{(\leq m+1)}| \geq |\mathcal{P}(m, s, \ell)^{(\leq m+1)}|
    \]
    we have $\nu(\mathcal{F}^{(m)}) < \ell$. Moreover, for any $k \leq \ell$ there is no $k$-matching $\{F_1, \ldots, F_k\}$ in $\mathcal{F}$ such that
    \[
    \sum_{i=1}^k |F_i| \leq km + k - \ell.
    \]
\end{lemma}

Let us first derive Lemma~\ref{l: matchings on layer m} from Lemma~\ref{l: parametrised matchings}.

\begin{proof}[Proof of Lemma~\ref{l: matchings on layer m} using Lemma~\ref{l: parametrised matchings}]
We need to check that for every fixed $m$ and all sufficiently large $s$ there exists a value of $t$ satisfying conditions \eqref{eq: condition t 1}, \eqref{eq: condition t 2}, and \eqref{eq: condition t 3}. We fix $m$ and define $t$ as the smallest integer satisfying the condition
\begin{equation} \label{eq: t definition}
\frac{1}{m!}\left(\frac{t}{m+1} + 1\right)\left(\frac{2m}{5}\right)^m >
\frac{(m+1)^{m-1}}{(m-1)!}.
\end{equation}
We now verify that the three conditions are satisfied for all sufficiently large $s$.

First, let us check \eqref{eq: condition t 3}. By Theorem~\ref{t: Frankl-Kupavskii EMC}, condition \eqref{eq: condition t 3} holds for $s \geq s_0$ and
$n \geq \frac{5}{3}m(\ell+t).$
The latter condition is satisfied provided $s \geq s_0(m)$ and $\ell \leq \frac{3}{5}s$, since
\[
\frac{5}{3}m(\ell+t) \leq ms + \frac{5}{3}mt = n - (s-\ell) + \frac{5}{3}mt < n.
\]
The last inequality holds for sufficiently large $s$, since $s-\ell \geq \frac{2}{5}s$ and $t$ is constant once $m$ is fixed.

Next, we check \eqref{eq: condition t 1}. We have
\begin{equation}\label{eq665}
    {(m+1)(s-\ell)-1 \choose m} > {\frac{2}{5}(m+1)s-1 \choose m},
\end{equation}
and
\begin{align*}
    &(\ell+t-1)\sum_{i=1}^{m-1}{n-1 \choose i-1} + \left({n - \ell + 1 \choose m} - {n - \ell - t \choose m}\right) \\
    &\leq \left(\frac{3}{5}s+t-1\right)\sum_{i=1}^{m-1}{n-1 \choose i-1} + (t+1){n - \ell \choose m - 1} \\
    &\leq \left(\frac{3}{5}s+t-1\right)\sum_{i=1}^{m-1}{(m+1)s \choose i-1} + (t+1){(m+1)s \choose m - 1}.
\end{align*}
For fixed $m$ and $t$, the right-hand side of \eqref{eq665} is a polynomial in $s$ of degree $m$. At the same time, the last displayed bound is a polynomial in $s$ of degree $m-1$. Thus, \eqref{eq: condition t 1} holds for all sufficiently large $s$.

Finally, we check \eqref{eq: condition t 2}. The left-hand side of \eqref{eq: condition t 2} is at least
\begin{equation} \label{eq: lhs condition t 2}
    \left(\frac{t}{m+1} + 1\right){\frac{2}{5}sm - tm-1 \choose m},
\end{equation}
and the right-hand side is at most
\begin{equation} \label{eq: rhs condition t 2}
    (s-1)\sum_{i=1}^{m}{(m+1)s \choose i-1}.
\end{equation}
For fixed $m$ and $t$, both \eqref{eq: lhs condition t 2} and \eqref{eq: rhs condition t 2} are polynomials in $s$ of degree $m$. We compare their leading coefficients. The leading coefficient of \eqref{eq: lhs condition t 2} is
\[
\frac{1}{m!}\left(\frac{t}{m+1} + 1\right)\left(\frac{2m}{5}\right)^m,
\]
and the leading coefficient of \eqref{eq: rhs condition t 2} is
\[
\frac{(m+1)^{m-1}}{(m-1)!}.
\]
Thus, by \eqref{eq: t definition}, the leading coefficient in \eqref{eq: lhs condition t 2} is larger than that in \eqref{eq: rhs condition t 2}, and therefore condition \eqref{eq: condition t 2} holds for all sufficiently large $s$.
\end{proof}

We emphasize that the constant $\frac{3}{5}$ in the condition $\ell \leq \frac{3}{5}s$ is needed only for the validity of \eqref{eq: condition t 3} and is inherited from Theorem~\ref{t: Frankl-Kupavskii EMC}. Therefore, any improvement of the constant $\frac{5}{3}$ in Theorem~\ref{t: Frankl-Kupavskii EMC} would immediately imply an improvement of the constant $\frac{3}{5}$ in Lemma~\ref{l: matchings on layer m}. In particular, if the Erd\H{o}s Matching Conjecture is proved for
\[
n \geq (1 + \varepsilon)sk, \qquad s \geq s_0(\varepsilon, k),
\]
then it would imply Lemma~\ref{l: matchings on layer m} for $\ell \leq (1 - \varepsilon)s$, provided $s \geq s_0(\varepsilon, m)$.

Next, we prove Lemma~\ref{l: parametrised matchings}.

\begin{proof}[Proof of Lemma~\ref{l: parametrised matchings}]
Obviously, it is enough to prove the moreover part, since it implies $\nu(\mathcal{F}^{(m)}) < \ell$. We argue indirectly. Let $\{F_1, \ldots, F_k\}$ be a matching such that $F_i \in \mathcal{F}$ for all $i$, $k \leq \ell$, and
\[
\sum_{i=1}^k |F_i| \leq km + k - \ell.
\]
Let
\[
X = [n] \setminus \bigcup_{i=1}^k F_i
\qquad\text{and}\qquad
\mathcal{F}' = \mathcal{F} \cap 2^{X}.
\]
Note that
\[
|X| \geq n - (km+k-\ell) = (s-k)(m + 1),
\]
and $\mathcal{F}'$ does not contain an $(s-k)$-matching. Thus, we may apply Lemma~\ref{l.frankl_shifting modified} and obtain
{\small
\[
{X\choose m+1}\setminus \mathcal{F}' \geq \left(1 + \frac{|X| - (m+1)(s-k)}{m+1}\right){|X| - 1 \choose m} \geq {(m+1)(s-k) - 1 \choose m}.
\]
}
Therefore,
{\small
\begin{equation} \label{eq: y(m+1) bound 1}
    y_{\mathcal{F}}(m+1)
    \geq {X\choose m+1}\setminus \mathcal{F}'
    \geq {(m+1)(s-k) - 1 \choose m}
    \geq {(m+1)(s - \ell) - 1 \choose m}.
\end{equation}
}

Next, we consider two cases, according to whether $\nu(\mathcal{F}^{(\leq m)}) < \ell + t$ or $\nu(\mathcal{F}^{(\leq m)}) \geq \ell + t$.

If $\nu(\mathcal{F}^{(\leq m)}) < \ell + t$, then by \eqref{eq: condition t 3},
\[
|\mathcal{F}^{(m)}| \leq e_{m}(n, \ell + t) = {n \choose m} - {n - \ell - t + 1 \choose m}.
\]
For $i < m$ we use Lemma~\ref{l.frankl_shifting} to get
\[
|\mathcal{F}^{(i)}| \leq e_{i}(n, \ell + t) \leq (\ell+t-1){n-1 \choose i - 1}.
\]
These bounds, together with \eqref{eq: y(m+1) bound 1} and \eqref{eq: condition t 1}, imply that
\[
|\mathcal{F}^{(\leq m+1)}| < |\mathcal{P}(m, s, \ell)^{(\leq m+1)}|.
\]
Indeed,
{\small
\begin{align*}
    &|\mathcal{P}(m, s, \ell)^{(\leq m+1)}| - |\mathcal{F}^{(\leq m+1)}| \geq y_{\mathcal{F}}(m+1) - \left(|\mathcal{F}^{(m)}| - |\mathcal{P}(m, s, \ell)^{(m)}| + \sum_{i=1}^{m-1}|\mathcal{F}^{(i)}|\right) \\
    &\geq {(m+1)(s-\ell)-1 \choose m} - \left({n - \ell + 1 \choose m} - {n - \ell - t \choose m} + (\ell+t-1)\sum_{i=1}^{m-1}{n-1 \choose i-1} \right).
\end{align*}
}
This quantity is positive by \eqref{eq: condition t 1}.

If $\nu(\mathcal{F}^{(\leq m)}) \geq \ell + t$, we use this to improve \eqref{eq: y(m+1) bound 1}. Let $G_1, \ldots, G_{\ell+t}$ be an $(\ell+t)$-matching in $\mathcal{F}^{(\leq m)}$. Define
\[
Y = [n] \setminus \bigcup_{i=1}^{\ell+t}G_i
\qquad\text{and}\qquad
\mathcal{G}' = \mathcal{F} \cap 2^{Y}.
\]
We have
\[
|Y| \geq n - (\ell+t)m = (s-\ell-t)(m + 1)+t,
\]
and $\mathcal{G}'$ does not contain an $(s-\ell-t)$-matching. Therefore, we again apply Lemma~\ref{l.frankl_shifting modified} and obtain
\begin{align}
    y_{\mathcal{F}}(m+1) \notag
    &\geq {Y\choose m+1}\setminus \mathcal{G}' \\
    \label{eq: y(m+1) bound 2}
    &\geq \left(1 + \frac{t}{m+1}\right){|Y| - 1 \choose m} \\
    \notag
    &\geq \left(1 + \frac{t}{m+1}\right){n - (\ell+t)m - 1 \choose m}.
\end{align}
To bound $|\mathcal{F}^{(i)}|$ for $i \leq m$, we use Lemma~\ref{l.frankl_shifting} together with the trivial bound $\nu(\mathcal{F}^{(i)}) < s$:
\begin{align} \label{eq: i-th layer bound}
    |\mathcal{F}^{(i)}| \leq e_{i}(n, s) \leq (s-1){n-1 \choose i - 1}.
\end{align}
Finally, combining \eqref{eq: y(m+1) bound 2} and \eqref{eq: i-th layer bound} with \eqref{eq: condition t 2}, we obtain
\begin{align*}
    &|\mathcal{P}(m, s, \ell)^{(\leq m+1)}| - |\mathcal{F}^{(\leq m+1)}| \geq y_{\mathcal{F}}(m+1) - \sum_{i=1}^{m}|\mathcal{F}^{(i)}| \\
    &\geq \left(1 + \frac{t}{m+1}\right){n - (\ell+t)m - 1 \choose m} - (s-1)\sum_{i=1}^{m}{n-1 \choose i - 1} > 0.
\end{align*}
\end{proof}

\section{Proof of Theorem~\ref{t: strong main}} \label{s: completing proof}

In this section, we complete the proof of Theorem~\ref{t: strong main}. Recall that it is enough to prove the theorem for shifted families. Assume that $n = sm + s-\ell$, $\ell \leq \left(\frac{1}{2} - \varepsilon\right)s$, $s \geq s_0(m, \varepsilon)$, and $\mathcal{F} \subset 2^{[n]}$ is a shifted family such that $\nu(\mathcal{F}) < s$ and
\[
|\mathcal{F}^{(\leq m + 1)}| \geq |\mathcal{P}(m, s, \ell)^{(\leq m+1)}|.
\]
We also assume that $\ell \geq 2$ and $m \geq 3$, since for $\ell = 1$ \cite{Kl} and for $m \leq 2$ \cite{KS} the value $e(sm+s-\ell, s)$ is known and the theorem is already established.

We will prove that
\[
\mathcal{F}^{(\leq m + 1)} \subset \mathcal{P}(m, s, \ell)^{(\leq m+1)}.
\]
The inclusion on the $(m+1)$-st layer is obvious since
$
\mathcal{P}(m, s, \ell)^{(m+1)} = {[n] \choose m + 1}.
$
Thus, we may and will exclude the $(m+1)$-st layer from the considerations below.
We begin by proving the inclusion on the $m$-th layer.

\begin{lemma} \label{l: mth layer inclusion} We have
    \[
    \mathcal{F}^{(m)} \subset \mathcal{P}(m, s, \ell)^{(m)}.
    \]
\end{lemma}

\begin{proof}
    We argue indirectly. By Lemma~\ref{l: matchings on layer m} we have $\nu(\mathcal{F}^{(m)}) < \ell$. Assume that
    $\mathcal{F}^{(m)} \not\subset \mathcal{P}(m, s, \ell)^{(m)}.$
    It is easy to check that for shifted $\mathcal{F}$ this implies that $\mathcal{F}^{(m)}$ is not isomorphic to a subfamily of $\mathcal{P}(m, s, \ell)^{(m)}$. Moreover, we have
    \[
    n \geq sm \geq \frac{1}{\frac{1}{2}-\varepsilon}\ell m > (2+\varepsilon)\ell m.
    \]
    Therefore, for $s \geq s_0(\varepsilon)$ we may apply Theorem~\ref{t: Hilton-Milner for EMC} and obtain
    \begin{align}
        |\mathcal{F}^{(m)}| \notag
        &\leq |\mathcal{A}_1^{(m)}(n,\ell)| - {n - m - \ell + 1 \choose m - 1} + 1 \\
        \label{eq: mth layer inclusion 1}
        &= |\mathcal{P}(m, s, \ell)^{(m)}| - {n - m - \ell + 1 \choose m - 1} + 1.
    \end{align}

    Applying the moreover part of Lemma~\ref{l: matchings on layer m} to $\mathcal{F}^{(i)}$ for $i < m$, we get
    \[
    i\nu(\mathcal{F}^{(i)}) > m\nu(\mathcal{F}^{(i)}) + \nu(\mathcal{F}^{(i)}) - \ell,
    \]
    or, equivalently,
    \[
    \nu(\mathcal{F}^{(i)}) < \frac{\ell}{m-i+1}.
    \]
    Thus, for $i \leq m-1$ we have $\nu(\mathcal{F}^{(i)}) < \frac{\ell}{2}$, and by Lemma~\ref{l.frankl_shifting},
    \begin{equation} \label{eq: mth layer inclusion 2}
        |\mathcal{F}^{(i)}| \leq \left(\left\lceil \frac{\ell}{2}\right\rceil - 1\right){n - 1 \choose i - 1} < \frac{\ell}{2}{n - 1 \choose i - 1}.
    \end{equation}
    Therefore,
    \begin{align*}
        |\mathcal{P}(m, s, \ell)^{(\leq m+1)}| - |\mathcal{F}^{(\leq m+1)}|
        &\ge |\mathcal{P}(m, s, \ell)^{(m)}| - |\mathcal{F}^{(m)}| +1 - \sum_{i=1}^{m-1}|\mathcal{F}^{(i)}| \\
        &\geq {n - m - \ell + 1 \choose m - 1}  - \frac{\ell}{2}\sum_{i=1}^{m-1} {n - 1 \choose i - 1} > 0.
    \end{align*}
    In the first inequality, we used the fact that $\mathcal{P}(m, s, \ell)^{(m-1)}\neq\emptyset$ for $m \geq 3$ and $\ell \geq 2$. In the second inequality, we used \eqref{eq: mth layer inclusion 1} and \eqref{eq: mth layer inclusion 2}. In the last inequality, we applied Lemma~\ref{l: HM calculation}.
\end{proof}

In the final step of the proof, we will need a simple observation about cross-dependent families. We say that families $\mathcal{F}_1, \ldots, \mathcal{F}_s$ are {\it cross-dependent} if there is no $s$-matching $\{F_1, \ldots, F_s\}$ such that $F_i \in \mathcal{F}_i$ for all $i\in [s]$.

\begin{lemma} \label{l: cross-dependent averaging}
Let $n \geq sk$ and let $\mathcal{F}_1, \ldots, \mathcal{F}_s\subset {[n]\choose k}$ be cross-dependent. Then
\begin{equation} \label{eq: cross-dependent averaging}
    \sum_{i=1}^{s}|\mathcal{F}_i| \leq (s-1){n \choose k}.
\end{equation}
\end{lemma}

Note that the bound in \eqref{eq: cross-dependent averaging} is best possible, since it is attained by taking $\mathcal{F}_i = {[n] \choose k}$ for $i\in[s-1]$ and $\mathcal{F}_s = \emptyset$.

\begin{proof}
    Let $\{F_1, \ldots, F_s\}$ be a random collection of pairwise disjoint sets from ${[n] \choose k}$. The condition $n \geq sk$ ensures that such collections exist. Let
    \[
    \xi = \sum_{i=1}^s \mathds{1}\{F_i \in \mathcal{F}_i\}.
    \]
    On the one hand, we have $\xi \leq s - 1$, since $\mathcal{F}_1, \ldots, \mathcal{F}_s$ are cross-dependent. On the other hand,
    \[
    \mathbb{E}\xi = \sum_{i=1}^s \frac{|\mathcal{F}_i|}{{n \choose k}}.
    \]
    Combining these bounds, we obtain \eqref{eq: cross-dependent averaging}.
\end{proof}

The following lemma completes the proof of Theorem~\ref{t: strong main}.

\begin{lemma} \label{l: lower layers inclusion}
    For all $i < m$ we have
    \[
    \mathcal{F}^{(i)} \subset \mathcal{P}(m, s, \ell)^{(i)}.
    \]
\end{lemma}

\begin{proof}
    We argue indirectly. Assume that there exists $F \in \mathcal{F}^{(i)}$ such that
    \[
    F \notin \mathcal{P}(m, s, \ell).
    \]
    If $i \leq m + 1 - \ell$, we immediately get a contradiction to Lemma~\ref{l: matchings on layer m} with $k = 1$ and $F_1 = F$. Thus, we may assume that
    \[
    i + \ell - m - 1 > 0.
    \]
    This inequality guarantees that all sums and integer intervals appearing below are nonempty.

    Recall that
    \[
    \mathcal{P}(m, s, \ell)^{(i)} = \left\{F \in {[n]\choose i}: |F \cap [\ell - 1]| \geq m - i + 1\right\}.
    \]
    Thus, we must have $|F \cap [\ell - 1]| \leq m - i$. Let
    \[
    x = |F \cap [\ell - 1]|
    \qquad\text{and}\qquad
    y = |F \cap [\ell, n]| = i - x.
    \]
    Since $\mathcal{F}$ is shifted, we have
    \[
    F' := [x] \cup [\ell, \ell + y - 1] \in \mathcal{F}.
    \]
    Put
    \[
    I := [x + 1,\, x + i + \ell - m - 1].
    \]
    Note that $I\subset [x+1,\ell-1]$, since
    $
    x+i+\ell-m-1 \le \ell-1
    $
    is equivalent to $x+i-m\le 0$. In particular, $I\cap F'=\emptyset$.

    For each $j\in I$, define
    \begin{align*}
        \mathcal{G}_{j} &= \big\{G \in \mathcal{F}^{(m)}: G \cap [\ell + m - 1] = \{j\} \big\},\\
        \mathcal{G}'_{j} &= \{G\setminus\{j\}: G \in \mathcal{G}_j\}.
    \end{align*}
    Note that $|\mathcal{G}_{j}| = |\mathcal{G}'_{j}|$ and
    \[
    \mathcal{G}'_{j} \subset {[\ell + m, n] \choose m - 1}.
    \]
    Also, note that $\mathcal{G}_j\subset \mathcal{P}(m,s,\ell)$. Moreover,
    \[
    F' \cap (I\cup[\ell + m, n]) = \emptyset,
    \]
    since $F'\cap I=\emptyset$ and  $y \leq i < m$. Therefore, $F' \cap G = \emptyset$ for every $j \in I$ and every $G \in \mathcal{G}_{j}$.

    Assume that the families $\mathcal G_j$, $j\in I$, are not cross-dependent. Then there exist sets $G_j \in \mathcal{G}_j$, $j \in I$, such that together with $F'$ they form an $(i + \ell - m)$-matching. But
    \[
    |F'| + \sum_{j \in I }|G_j|
    = i + (i + \ell - m - 1)m
    = (i + \ell - m)m + (i + \ell - m) - \ell.
    \]
    This contradicts Lemma~\ref{l: matchings on layer m} with $k = i + \ell - m$.

    Thus, we may assume that the families $\mathcal{G}'_j$, $j\in I$, are cross-dependent. We apply Lemma~\ref{l: cross-dependent averaging}. Let us check that the requirement $n \geq sk$ of that lemma is satisfied. We have $i + \ell - m - 1 \leq \ell - 1$ families in
    $
    {[\ell + m , n] \choose m - 1},
    $
    and
    \[
    |[\ell + m, n]| = n - m-\ell + 1 \ge sm  - m - \ell + 1  \geq (\ell - 1)(m-1).
    \]
    Therefore, Lemma~\ref{l: cross-dependent averaging} yields
    \begin{equation} \label{eq: averaging application}
    \sum_{j\in I}|\mathcal{G}_j|
    = \sum_{j\in I}|\mathcal{G}'_j|
    \leq (i + \ell - m - 2){n - m - \ell +1 \choose m - 1}.
    \end{equation}

    There are
    \[
    (i + \ell - m - 1){n - m - \ell +1 \choose m - 1}
    \]
    sets $G \in {[n] \choose m}$ such that
    $
    G \cap [\ell + m - 1] = \{j\}
    $
    for some $j \in I$. All of them belong to $\mathcal{P}(m, s, \ell)$. At the same time, by \eqref{eq: averaging application}, at least
    $
    {n - m - \ell + 1 \choose m - 1}
    $
    of them do not belong to $\mathcal{F}$. Also, recall that by Lemma~\ref{l: mth layer inclusion} we have
    \[
    \mathcal{F}^{(m)} \subset \mathcal{P}(m, s, \ell)^{(m)}.
    \]
    Altogether, this implies
    \begin{align} \label{eq: small mth layer}
        |\mathcal{P}(m, s, \ell)^{(m)}| - |\mathcal{F}^{(m)}| \geq {n - m - \ell + 1 \choose m - 1}.
    \end{align}

    We now conclude as in the proof of Lemma~\ref{l: mth layer inclusion}. By Lemmas~\ref{l: matchings on layer m} and~\ref{l.frankl_shifting} we have \eqref{eq: mth layer inclusion 2}. Combining \eqref{eq: small mth layer} and \eqref{eq: mth layer inclusion 2}, we get
    \begin{align*}
        |\mathcal{P}(m, s, \ell)^{(\leq m+1)}| - |\mathcal{F}^{(\leq m+1)}|
        &\geq |\mathcal{P}(m, s, \ell)^{(m)}| - |\mathcal{F}^{(m)}| - \sum_{i=1}^{m-1}|\mathcal{F}^{(i)}| \\
        &\geq {n - m - \ell + 1 \choose m - 1} - \frac{\ell}{2}\sum_{i=1}^{m-1} {n - 1 \choose i - 1} > 0.
    \end{align*}
    In the last inequality, we used Lemma~\ref{l: HM calculation}.
\end{proof}

\end{document}